\documentclass[12pt, letterpaper]{article} 
\pdfoutput=1

\newcommand{\papertitle}{{Relating $p$-adic eigenvalues and the local
  Smith normal form}}

\newcommand{\paperauthors}{{Mustafa Elsheikh, Mark Giesbrecht}}

\usepackage{amssymb} 
\usepackage{amsmath} 
\usepackage{amsthm}
\usepackage{latexsym}
\usepackage[round]{natbib}
\usepackage[svgnames]{xcolor}
\usepackage{comment}
\usepackage{hyphenat}
\usepackage[pdfauthor=\paperauthors, pdftitle=\papertitle, colorlinks,linkcolor=darkblue,citecolor=darkgreen,urlcolor=darkred]{hyperref}
\usepackage{tikz}
\usepackage[textwidth=8em,backgroundcolor=yellow,textsize=footnotesize]{todonotes}
\usepackage{enumitem}

\definecolor{darkgreen}{rgb}{0,.4,0}
\definecolor{darkblue}{rgb}{0,0,.5}
\definecolor{darkred}{rgb}{.6,0,0}

\numberwithin{equation}{section}
\newtheorem{theorem}{Theorem}
\newtheorem*{theorem*}{Theorem}
\newtheorem{lemma}{Lemma}

\newtheorem{definition}{Definition}
\newtheorem{fact}{Fact}
\newtheorem{example}{Example}

\renewcommand{\det}{\operatorname{det}}

\DeclareMathOperator{\sgn}{sgn}
\DeclareMathOperator{\rank}{rank}
\DeclareMathOperator{\diag}{diag}
\DeclareMathOperator{\orbit}{orbit}
\DeclareMathOperator{\stab}{stab}

\DeclareMathOperator{\rem}{rem}

\DeclareMathOperator{\nequiv}{\mskip4mu\not\equiv\mskip4mu}

\DeclareMathOperator{\ord}{{\mathit{v}}}

\DeclareMathOperator{\cont}{cont}

\newcommand{\K}{\mathsf{K}}

\newcommand{\Abar}{\overline{A}}
\newcommand{\kbar}{\overline{k}}

\newcommand{\Ahat}{\widehat{A}}
\newcommand{\rhat}{\hat{r}}

\newcommand{\Jhat}{\widehat{J}}

\newcommand{\Atil}{\widetilde{A}}

\newcommand{\NP}{{\mathsf{NP}}}
\newcommand{\GL}{{\mathrm{GL}}}
\newcommand{\M}{{\mathrm{M}}}

\newcommand{\ZZ}{{\mathbb{Z}}}
\newcommand{\QQ}{{\mathbb{Q}}}

\newcommand{\RR}{{\mathbb{R}}}
\newcommand{\NN}{{\mathbb{N}}}

\newcommand{\R}{{\mathsf{R}}}

\newcommand{\nxn}{{n\times n}}

\newcommand{\calC}{{\mathcal{C}}}

\renewcommand{\SS}{{\mathfrak{S}}}

\newcommand{\MnZ}{{\ZZ^\nxn}}

\newcommand{\sbar}{{\overline{s}}}
\newcommand{\gbar}{{\overline{g}}}
\newcommand{\Deltabar}{{\overline{\Delta}}}

\newcommand{\y}{y}

\newcommand{\U}{\mathcal{U}}
\newcommand{\V}{\mathcal{V}}

\begin{document}


\title{\papertitle}
\author{\paperauthors \\
  {melsheik@uwaterloo.ca, mwg@uwaterloo.ca}\\
  {Cheriton School of Computer Science,}\\
  {University of Waterloo, Canada}}


\maketitle

\begin{abstract}
  Conditions are established under which the $p$-adic valuations of
  the invariant factors (diagonal entries of the Smith form) of an
  integer matrix are equal to the $p$-adic valuations of the
  eigenvalues.  It is then shown that this correspondence is the
  typical case for ``most'' matrices; density counts are given
  for when this property holds, as well as easy transformations to this
  typical case.
\end{abstract}


\subsection*{Keywords}
Integer matrices, $p$-adic numbers, eigenvalues, Smith normal form

\subsection*{AMS Subject Classification}
15A36, 15A18, 15A21

\section{Introduction}
\label{sec:intro}

Recall that any matrix $A\in\MnZ$ of rank $r$ can be written as $A=PSQ$ where
$P,Q\in\MnZ$ are unimodular matrices (i.e., whose inverses are also in $\MnZ$)
and $S=\diag(s_1,\ldots,s_r,0,\ldots,0)$ is the \emph{Smith normal form} (SNF)
of $A$, where $s_1,\ldots,s_r\in\ZZ$ are $A$'s \emph{invariant factors},
and $ s_1 \mid s_2
  \mid \cdots \mid s_r $.  Alternatively, if we define the $i$th
\emph{determinantal divisor} $\Delta_i$ of $A$ as the GCD of all $i\times i$
minors of $A$, then $\Delta_{i-1}$ divides $\Delta_i$ and $s_1=\Delta_1$ and
$s_i=\Delta_i/\Delta_{i-1}$ for $2\leq i\leq r$.
See \citep{Newman:1972} for a full treatment of this theory.

\emph{A priori} the invariant factors of a matrix and the
\emph{eigenvalues} of a matrix would seem to be rather different
invariants, the former related to the $\ZZ$-lattice structure of $A$
and the latter to the geometry of the linear map. We show that, in
fact, they are ``usually'' in one to one correspondence with respect
to their $p$-adic valuations at a prime $p$.  We demonstrate a simple
sufficient condition under which this holds for any integer matrix,
and provide bounds on the density of matrices for which it holds.  The
list of powers of $p$ in the invariant factors are often referred to as
the \emph{local Smith form at $p$} by some authors \citep{Gerstein:1977,
Dumas:2001, Wilkening:2011, Elsheikh:2012}.

Throughout we will work in the ring of $p$-adic integers
$\ZZ_p\supseteq\ZZ$, the $p$-adic completion of $\ZZ$, and its
quotient field $\QQ_p\supseteq\QQ$, the $p$-adic numbers.  See
\citep{Koblitz:1984} or \citep{Gouvea:1997} for an introduction.  Let
$\ord_p(a) \in \NN \cup \{ \infty \}$ be the \emph{$p$-adic order} or \emph{$p$-adic
  valuation} of any $a \in \ZZ_p$, the number of times $p$ divides $a$
exactly, where $v_p(0) = \infty$. The valuation can be extended to $\QQ_p$ by
letting $v_p(a/b)=v_p(a)-v_p(b)$ for
$a, b\in\ZZ_p$.

\begin{example}\label{ex:good-matrix-intro}
  Consider the matrix
  \[ 
  A =
  \begin{pmatrix} 3 & -1 & 3 \\
    9 & -10 & 0 \\ 3 & 0 & 3
  \end{pmatrix} = \overbrace{\begin{pmatrix} 1&-1&0\\ 1&0&1\\ 0&-1&0
    \end{pmatrix}}^U \overbrace{\begin{pmatrix} 1& &\\ &3&\\ &&9
    \end{pmatrix}}^S \overbrace{\begin{pmatrix} 0&-1&0\\ -1&0&-1\\ 1&-1&0
    \end{pmatrix}}^V,
  \] 
  for unimodular $U,V$ and Smith form $S$ of $A$.  Now consider the
  eigenvalues of $A$, which are roots of the characteristic polynomial
  \[ 
  f = \det(xI - A) = x^3+4x^2-51x-27\in\ZZ[x]. 
  \] 
  We find it has three distinct roots in $\ZZ_3$:
  \begin{align*}
    \lambda_1 & = -1 -3^3-3^4-3^5-3^6-3^8+3^9+O(3^{10}),\\
    \lambda_2 & = -3 -3^2 -3^3-3^4+3^6-3^8-3^9 + O(3^{10}),\\
    \lambda_3 & = \mskip14mu 3^2 -3^3-3^5+3^6-3^8+3^9 + O(3^{10}).
  \end{align*}
  In this example we see that $\ord_3(\lambda_1) = 0$,
  $\ord_3(\lambda_2) = 1$ and $\ord_3(\lambda_3) = 2$.  Recalling that the Smith
  form of $A$ above is $S = \diag(1, 3, 3^2)$, we see that the diagonal entries
  of the Smith form have precisely the same $p$-adic valuations as the
  eigenvalues of $A$. \qed
\end{example}

The eigenvalues of $A$ are roots of the characteristic polynomial,
which has a natural image in $\ZZ_p[x]$ since $\ZZ_p$ contains $\ZZ$.
Thus, the eigenvalues of $A$ can naturally be viewed as
$p$-adic algebraic integers in a finite-degree algebraic
extension field $\K_p$
over $\QQ_p$ \cite[Proposition 5.4.5 (v)]{Gouvea:1997}.

In order to show the correspondence between the eigenvalues and
the invariant factors,
we need to extend the definition of the valuation $v_p$ to the eigenvalues
(more generally, to the elements of $\K_p$).
If an element $a \in \K_p$ has a minimal polynomial
$x^{d_a} + a_{d_a-1} x^{d_a-1} + \ldots + a_0 \in \QQ_p[x]$,
then the valuation is uniquely given by
$v_p(a) = (1/d_a) v_p(a_0)$.
See \cite[\S 3, pp. 66]{Koblitz:1984}.
The image of the extended $v_p$ is $\QQ$, and its restriction to $\QQ_p$
agrees with the earlier definition of $v_p$ on $\QQ_p$.
The valuation of a non-zero eigenvalue $v_p(\lambda_i)$
is independent of the choice of $\K_p$,
since it only depends on the minimal polynomial of $\lambda_i$ over $\QQ_p$.
In particular, the set of minimal polynomials of the non-zero eigenvalues
is precisely the set of irreducible factors of the
characteristic polynomial of the matrix over $\QQ_p$ regardless of the
field extension. Therefore, $v_p(\lambda_1), \ldots, v_p(\lambda_n)$
are invariants of the matrix over $\QQ_p$, and independent of the $p$-adic
extension chosen to contain the eigenvalues.



  In light of the above, we will treat integer
matrices and their eigenvalues as being naturally embedded in $\ZZ_p$,
$\QQ_p$ or $\K_p$ as appropriate, under the $p$-adic valuation $v_p$.

It should be noted that the correspondence between the valuations of the
eigenvalues and the invariant factors does not hold for all matrices. 

\begin{example}\label{ex:bad-matrix-intro}
Let
  \[ 
  A = \begin{pmatrix}
    37 & 192 & 180 & 369 \\
    55 & 268 & 198 & 531 \\
    163 & 758 & 442 & 1539 \\
    198 & 908 & 486 & 1858
  \end{pmatrix}, 
  \] 
  which has the Smith form decomposition:
  \[ 
  A = \begin{pmatrix}
    1 & 1 & 1 & 1 \\
    1 & 0 & 1 & 0 \\
    1 & 0 & 0 & 0 \\
    0 & 1 & 0 & 0
  \end{pmatrix}
  \begin{pmatrix}
    1 &  &  &  \\
     & 2 &  &  \\
     &  & 2 &  \\
     &  &  & 4
  \end{pmatrix}
  \begin{pmatrix}
    163 & 758 & 442 & 1539 \\
    99 & 454 & 243 & 929 \\
    -54 & -245 & -122 & -504 \\
    -54 & -246 & -126 & -505
  \end{pmatrix}.
  \]
  The characteristic polynomial of $A$ is 
  \[ 
  f = x^{4} - 2605x^{3} + 39504x^{2}
  + 40952x + 16 \in \ZZ[x] 
  \] 
  which factors over $\QQ_2$ into 
  \[ 
  x + \left(1 + 2^2 + 2^3 + O(2^5)\right)\in\ZZ_2[x]
  \] 
  and the irreducible factor 
  \[ 
  x^3 + \left(2^3 +  O(2^5)\right) x^2 +\left (2^3 + 2^4 +
    O(2^5)\right) x + \left(2^4 + O(2^5)\right) \in\ZZ_2[x]. 
  \] 
  Using Newton slopes (see Fact~\ref{fact:newtonp} below) we find that the
  $2$-adic valuations of the roots of the second factor are $[4/3,
  4/3, 4/3]$.  Thus, the $2$-adic valuations of the invariant
  factors of $A$ are $[0, 1, 1, 2]$, while the $2$-adic valuations of the
  eigenvalues of $A$ are $[0, 4/3, 4/3, 4/3]$. The
  $p$-adic valuations of the
  eigenvalues and
  the invariant factors are therefore \emph{not} in $1$-$1$ correspondence. \qed
\end{example}

In the remainder of this paper we explore the conditions under which this
correspondence between the $p$-adic valuation of invariant factors and
eigenvalues occurs, and show that it is, in fact, 
the  ``typical'' case, i.e.,  it
holds for ``most'' matrices.

\subsection{Main Results}
\label{sec:main-result}

We first define two important matrix properties for our purposes.  

\begin{definition}
  Let $A\in \MnZ$ be of rank $r$ and $p$ be any prime.  Assume
  \begin{itemize} 
  \item[(i)] $A$ has Smith form $S=\diag(s_1,\ldots,s_r,0,\ldots,0)$
    over $\ZZ$,
    so that $\Delta_i=s_1\cdots s_i$ is the $i$th determinantal
    divisor of $A$, for $1\leq i\leq r$;
  \item[(ii)] $A$ has non-zero eigenvalues (with multiplicity)
    $\lambda_1,\ldots,\lambda_r$ in a finite-degree extension $\K_p$ over
    $\QQ_p$, and assume that
    $\ord_p(\lambda_1)\leq \cdots \leq\ord_p(\lambda_r)$;
  \item[(iii)] $A$ has characteristic polynomial $f = x^n + f_{1} x^{n-1} +
    \ldots + f_{r} x^{n - r} \in \ZZ[x]$ (note the reversed indexing).
\end{itemize}
\smallskip
\noindent
We say $A$ is \emph{$p$-characterized} if and only if
$\ord_p(f_i) = \ord_p(\Delta_i)$ for all $i \in [1, r]$.
\smallskip
\noindent
We say $A$ is \emph{$p$-correspondent} if and only if
$\ord_p(s_i)=\ord_p(\lambda_i)$ for all $i \in [1, r]$.
\end{definition}

Note that if $A$ is $p$-correspondent, then the valuations of the eigenvalues
are non-negative integers (since $v_p(s_i) \ge 0$).
Our main goal is to study the notion of $p$-correspondence; that is the
relationship between the spectrum and the invariant factors.  The notion of
$p$-characterization is an auxiliary definition used throughout our
proofs.  The following theorem gives the relationship between a
matrix being $p$-correspondent and being $p$-characterized; the proof is in
Section~\ref{sec:proof}.

\begin{theorem}
  \label{thm:pdet->pcosp}
  Let $A \in \MnZ$ and $p$ be a prime.  If $A$ is $p$-characterized then $A$
  is $p$-correspondent.
\end{theorem}

Of course, not all matrices are $p$-correspondent at any particular
prime $p$, but it is generally possible to transform a matrix to a
$p$-correspondent one.  We offer the following two simple lemmas in this
regard, the proofs of which are in Section \ref{sec:proof}.

\begin{lemma}
  \label{lem:exists-equiv}
  Let $A\in\MnZ$ and $p$ be any prime. There exists an equivalence transformation
  $P,Q \in \GL_n(\ZZ)$ such that $PAQ$ is $p$-correspondent.
\end{lemma}

\begin{lemma}
  \label{lem:exists-sim}
  Let $A \in \MnZ$ be non-singular, $p$ be any prime.  There exists a similarity
  transformation $U$ with entries in an extension $\K_p$ over $\QQ_p$ such that
  $U^{-1}AU$ is $p$-correspondent.
\end{lemma}

In Section \ref{sec:density} we establish that ``most''
matrices are $p$-correspondent.  We will consider the density in each
equivalence class defined by a given Smith form $S$.  The following
definition helps capture this.

\begin{definition}
  \label{def:SS_S^m}
  Fix a prime $p$, positive integers $m,n$, and integers $0 \leq e_1\leq
  e_2\leq \cdots \leq e_n$.  Let
  $S=\diag(p^{e_1},\ldots,p^{e_n})\in\ZZ^\nxn$.  Define $\SS_S^m\subseteq\ZZ^\nxn$ as
  the set of integer matrices with entries from $[0, p^m)$ whose Smith
  form $\diag(s_1, \ldots, s_n)$ satisfies $v_p(s_i) = e_i$ for all $i
  \in [1, n]$. 
\end{definition}

Our main result is then as follows.

\begin{theorem}
  \label{thm:cospdense}
  Let $n$ be a positive integer, $\epsilon > 0$, and $p$ any prime greater than
  $16 (n^2+3n)/ \epsilon$. Fix a set of integers $0 \le e_1 \le e_2 \le
  \cdots \le e_n$ and let $m \ge e_1 + \ldots + e_n + 1$ and $S =
  \diag(p^{e_1}, \ldots, p^{e_n}) \in \ZZ^{\nxn}$.  Then the number of matrices
  in $\SS_S^m$ which are $p$-characterized and hence
  $p$-correspondent is at least $(1 - \epsilon) \cdot
  |\SS_S^m|$.
\end{theorem}
  
\subsection{Previous Work}

\cite{Newman:1991}, Section 8, study the relationship between eigenvalues
and invariant factors of matrices over rings of algebraic
integers. Their results are concerned with products of eigenvalues
rather than individual eigenvalues (or subsets thereof). For
any square matrix over a ring $\R$ of algebraic integers
with invariant factors $s_1, \ldots, s_n$ and
eigenvalues $\lambda_1, \ldots, \lambda_n$ (in some
extension\footnote{
As stated in \cite[Section 2]{Newman:1991} $\bar{\R}$
is taken to be a ring of algebraic integers which contains $\R$ such that every
ideal generated within $\R$
becomes principal within $\bar{\R}$.}
$\bar{\R}$ of $\R$), they prove
(in Theorem 6)
 that for all $k
\in \{1,\ldots, n\}$ and all indexing sets $I \subseteq \{ 1, \ldots,
n \}$, $|I| = k$,
\begin{equation*}
  \label{eq:NT91-thm6} 
  s_1 s_2 \cdots s_k \mid \prod_{i \in I} \lambda_i\,.
\end{equation*}
where divisibility is taken over $\bar{\R}$.

\cite{Rushanan:1995} studied the Smith form and spectrum of
non-singular matrices with integer entries. He
established divisibility relations between the largest invariant
factor $s_n$ and the product of all eigenvalues. 
Recently, the connection between the eigenvalues and Smith form has
also been studied by \cite{Kirkland:2007} for integer matrices with
integer eigenvalues arising from the Laplacian of graphs, and by
\cite{Lorenzini:2008} for Laplacian matrices of rank $n-1$.


\section{Establishing $p$-correspondence}
\label{sec:proof}

We proceed to prove Theorem \ref{thm:pdet->pcosp}, that all
$p$-characterized matrices are $p$-correspondent. 
First recall that the coefficients of the characteristic polynomial
$f = x^n + \sum_{1\leq i\leq n}f_ix^{n-i}\in\ZZ[x]$ of a matrix $A\in\MnZ$ are
related to the minors of $A$.  For $1 \leq i \leq n$, let
$\calC_i^{n}$ denote the set of all $i$-tuples of integers of the form
$t = (t_1, \ldots, t_i)$ where $1 \leq t_1 < \cdots < t_i \leq n$. For
$\sigma, \tau \in \calC_i^{n}$, let $A \binom{\sigma}{\tau}$ denote
the determinant of the $i\times i$ submatrix selected by rows
$\sigma_1, \ldots, \sigma_i$ and columns $\tau_1, \ldots, \tau_i$;
this is the \emph{minor} of $A$ selected by $\sigma$ and $\tau$.
It is well-known and easily derived that, for all $1 \le i \le n$,
\begin{equation} 
  \label{eq:fi-sum-minors} 
  f_i = (-1)^i \sum_{\sigma \in \calC_i^n} A\binom{\sigma}{\sigma}.
\end{equation}
Since $\Delta_i$ divides all $i\times i$ minors, we have
$\Delta_i \mid f_i$, i.e., $\ord_p(f_i) \geq \ord_p(\Delta_i)$.
Moreover, if $A$ has rank $r$ we have $f_{r+1}=f_{r+2}=\cdots=f_n=0$.

We will also require the so-called Newton polygon of the
characteristic polynomial of $A$.
For a polynomial $f = x^n + \sum_{1\leq i\leq n} f_i x^{n-i} \in \ZZ_p[x]$,
the \emph{Newton polygon} of $f$, denoted by
$\NP(f)$, is the lower convex hull of the following points in $\RR^2$:
$\{ (0, 0), (1, \ord_p(f_1)),$ $\ldots$, $(n, \ord_p(f_n)) \}$
(we omit $(i, \ord_p(f_i))$ whenever $f_i = 0$).
This hull is represented by a list of points
$(x_1,y_1),\ldots,(x_{k},y_{k})\in\RR^2$ with
$x_1<x_2<\ldots<x_{k}$.  For each segment of $\NP(f)$ connecting two
adjacent points $(x_{i-1}, y_{i-1})$ and $(x_{i}, y_{i})$, the
\emph{slope} of the segment is $m_i = (y_{i} - y_{i-1})/(x_{i} -
x_{i-1})$ and the \emph{length} of the segment
 is the length of its projection onto the
$x$-axis, taken as $\ell_i = x_i -
  x_{i-1}$. An important use of this is the following.

\begin{fact}[See \cite{Koblitz:1984}, \S IV.3, Lemma 4]
  \label{fact:newtonp}
  Let $f = x^n + f_{1} x_{n-1} + \ldots + f_n \in \ZZ_p[x]$ 
  and $f_n \neq 0$. Let the roots of $f$ (counting multiplicity) be
  $\lambda_1,\ldots,\lambda_{n}$ in
  an extension $\K_p$ over $\QQ_p$.
  If the Newton polygon of $f$ has slopes $m_1,\ldots,m_k$ and
  lengths $\ell_1,\ldots,\ell_k$ as above, then for each $1\leq j\leq k$, $f$
  has exactly $\ell_j$ roots $\lambda\in\K_p$ whose valuation
  $\ord_p(\lambda)=m_j$.
\end{fact}
  
We now have all the tools to prove Theorem \ref{thm:pdet->pcosp}.

\setcounter{theorem}{0}
\begin{theorem}
  Let $A \in \MnZ$ and $p$ a prime.  If $A$ is $p$-characterized then $A$
  is $p$-correspondent.
\end{theorem}

\begin{proof}
  Assume that $A$ is $p$-characterized with rank $r$ and characteristic
  polynomial $f = \sum_{0\leq i\leq r} f_i x^{n-i} \in \ZZ[x]$, and
  $A$ has Smith form $S = \diag(s_1,\ldots,s_r,$ $0,\ldots,0)\in\MnZ$.
  Also, assume that the $p$-adic valuations of the invariant factors
  $s_1,\ldots,s_r$ have multiplicities $r_0,\ldots,r_{e-1}$ as
  follows:
  \[
  (\ord_p(s_1), \ldots, \ord_p(s_r)) = (\underbrace{0,\ldots,0}_{r_0},
  \underbrace{1,\ldots,1}_{r_1}, \ldots , \underbrace{e-1,\ldots,e-1}_{r_{e-1}}),
  \]
  where $e=\ord_p(s_r)+1$.  Since $A$ is $p$-characterized, by definition
  we have for $1\leq i\leq r$ that
  \[
  \ord_p(f_i) = \ord_p(\Delta_i) = \sum_{1\leq j\leq i} \ord_p(s_j),
  \]
  for all $1\leq i\leq r$.
  For notational convenience, define $m_i$ as
  \[
   m_i=\ord_p(\Delta_{r_0+r_1+\cdots+r_i}) = r_1+2r_2+\cdots + i\cdot
   r_i \ .
  \]
  Grouping the non-zero coefficients of $f$ by $p$-adic
  valuation we then have
  \begin{align*}
    (\ord_p(f_1), & \ldots,\ord_p(f_r)) \\
           & = \Bigl( 
    \underbrace{0,\ldots,0}_{r_0},\
    \underbrace{1,2,3,\ldots,r_1}_{r_1},\
    \underbrace{m_1+2, m_1+4, \ldots, m_1+2r_2}_{r_2},\\
    & \hspace*{23pt} \ldots,\  \underbrace{m_{e-2}+(e-1),\ 
     m_{e-2}+2(e-1),
      \ldots,\ m_{e-2}+r_{e-1}}_{r_{e-1}}(e-1)\ \Bigr).
  \end{align*}
  $\NP(f)$ is easily seen to consist of $e$ segments, where segment $i$
  has slope $i$, and length $r_i$, for $0\leq i<e$
  (a segment $i$ may have length 0 if $r_i=0$). Thus, by Fact
  \ref{fact:newtonp}, $f$ has $r_i$ roots $\lambda$ with
  $\ord_p(\lambda)=i$.  This accounts for all the non-zero roots of
  $f$, since $r_0+r_1+\cdots+r_{e-1}=\rank(A)$. Since these roots are
  the non-zero eigenvalues of $A$, we immediately see that $A$ is
  $p$-correspondent.
\end{proof}

It should be noted that the converse of Theorem~\ref{thm:pdet->pcosp} is not
necessarily true. The matrix in the following example is $p$-correspondent but
not $p$-characterized.


\begin{example}\label{ex:converse}
  The invariant factors of
  \[ A =
  \begin{pmatrix}
    -20 & -2 & 81 & -388 \\
    18 & -6 & -84 & 375 \\
    7 & 34 & 3 & 41 \\
    13004 & -11695 & -64944 & 289315
  \end{pmatrix},
  \]
  are $[1, 3, 3, 9]$, and the $3$-adic eigenvalues are: 
  \[ 
  2 + O(3),\  2 \cdot 3 +
  O(3^3),\ 3 + O(3^2),\ 3^2 + O(3^3). 
  \] 
  However, the $3$-adic valuation of the
  determinantal divisors is $[0, 1, 2, 4]$ and the characteristic polynomial
  over $\ZZ_3[x]$ is: 
  \[ 
  x^4 + (1 + O(3)) x^3 + (2 \cdot 3^2 + O(3^6)) x^2 + (2
  \cdot 3^2 + O(3^3)) x + (3^4 + O(3)). 
  \] 
  This is due to the fact that the Newton
  polygon of $A$ is the \emph{convex hull} of the segments defined by the
  coefficients of characteristic polynomial.
  
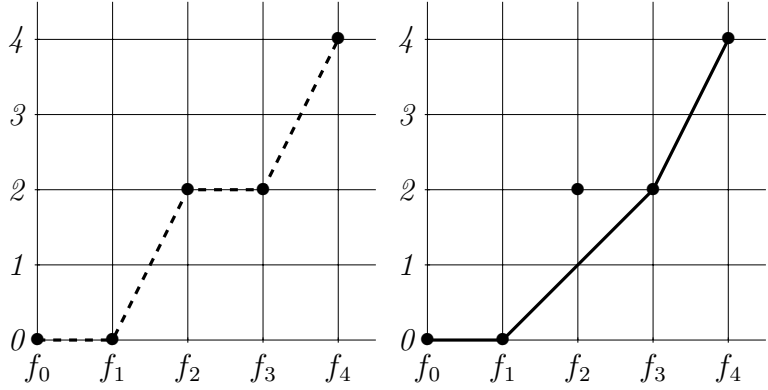
\begin{center}
    \begin{tikzpicture}
      \draw[step=1cm, very thin] (0, 0) grid (4.5, 4.5);
      \foreach \x in {0, 1, 2, 3, 4}
      \draw (\x cm, 1pt) -- (\x cm, -1pt) node[anchor = north] {$f_\x$};
      \foreach \y in {0, 1, 2, 3, 4}
      \draw (1pt, \y cm) -- (-1pt, \y cm) node[anchor = east] {\y};
      \foreach \point in {(0, 0), (1, 0), (2, 2), (3, 2), (4, 4)}
      \node at \point {$\bullet$};
      \draw[very thick, dashed] (0, 0) -- (1, 0) -- (2, 2) -- (3, 2) -- (4, 4);
    \end{tikzpicture}
    \begin{tikzpicture}
      \draw[step=1cm, very thin] (0, 0) grid (4.5, 4.5);
      \foreach \x in {0, 1, 2, 3, 4}
      \draw (\x cm, 1pt) -- (\x cm, -1pt) node[anchor = north] {$f_\x$};
      \foreach \y in {0, 1, 2, 3, 4}
      \draw (1pt, \y cm) -- (-1pt, \y cm) node[anchor = east] {\y};
      \foreach \point in {(0, 0), (1, 0), (2, 2), (3, 2), (4, 4)}
      \node at \point {$\bullet$};
      \draw[very thick] (0, 0) -- (1, 0) -- (3, 2) -- (4, 4);
    \end{tikzpicture}
  \end{center}
  While the coefficients of the characteristic polynomial (points in left
  figure) do not correspond to the $3$-adic valuations of the determinantal
  divisors, their lower convex cover (segments in right figure) corresponds to
  the $3$-adic valuations of the invariant factors with slopes: $0$, $1$
  (twice), and $2$. \qed
\end{example}

We now prove the two simple Lemmas \ref{lem:exists-equiv} and
\ref{lem:exists-sim}, establishing $p$-correspondence under unimodular
equivalence transformations and similarity.

\setcounter{lemma}{0}
\begin{lemma}
  Let $A\in\MnZ$ and $p$ be any prime. There exists an equivalence transformation
  $P,Q \in \GL_n(\ZZ)$ such that $PAQ$ is $p$-correspondent.
\end{lemma}
\begin{proof}
  Simply choose $P,Q \in \GL_n(\ZZ)$ such that $PAQ$ is in Smith normal form
  $S=\diag(s_1,\ldots,s_r,0,\ldots,0)$.  Then the eigenvalues of $PAQ$ are
  $s_1,\ldots,s_r$.
\end{proof}

\begin{lemma}
  Let $A \in \MnZ$ be non-singular, $p$ be any prime.  There exists a similarity
  transformation $U$ with entries in an extension $\K_p$ over $\QQ_p$ such that
  $U^{-1}AU$ is $p$-correspondent.
\end{lemma}
\begin{proof}
  Choose $\K_p$ to be a splitting field of the minimal polynomial of $A$.
  It is well-known that $A$ is similar to a matrix $J \in \K_p^{n \times n}$ in
  Jordan form.  That is, there exists an invertible $W \in \K_p^{n \times n}$ such
  that $W^{-1}AW = \diag(J_1,\ldots,J_\ell)$ where
  \[
  J_i = \begin{pmatrix}
    \mu_i & 1 & \\
    & \ddots & \ddots\\
    & & \ddots & 1\\
    &        & & \mu_i
  \end{pmatrix},
  \]
  for some (not necessarily unique) eigenvalue $\mu_i \in \K_p$ of $A$,
  and $J_i$ has dimensions $k_i \times k_i$.
  However, we can choose an alternative Jordan block $\Jhat_i$,
  similar to $J_i$, by applying
  the similarity transformation 
    $\diag(1, 1/\mu_i, \ldots, 1/\mu_i^{k_i-1})$ to $J_i$ to get
  \[
  \Jhat_i = \begin{pmatrix}
    \mu_i & \mu_i & \\
    & \ddots & \ddots\\
    & & \ddots & \mu_i\\
    &        & & \mu_i
    \end{pmatrix}.
  \]
  The Smith
    form of $\Jhat_i$ can obtained as follows. Subtract the first column from
    the second column. Then subtract the second column from the third, and so
    forth. The resulting matrix is $\diag(\mu_i,\ldots,\mu_i)$ which is in
    Smith normal form. Therefore $\Jhat_i$ is $p$-correspondent.

    Combining together the different Jordan blocks to form an
    alternative Jordan form $\Jhat$ for $A$, we see that $\Jhat$ is
    $p$-correspondent, and similar to $A$, as required.
\end{proof}

Note that if $A$ is singular, Lemma \ref{lem:exists-sim} may not
hold.  Consider for example
\[
A=\begin{pmatrix}
  0 & 1\\
  0 & 0
\end{pmatrix},
\]
whose only eigenvalue is zero, with multiplicity two.
This is also the case for any matrix similar to $A$.
  However, this
matrix has rank one, and so one of the invariant factors must always
be non-zero.


\section{Density of $p$-characterized matrices}
\label{sec:density}

This section provides the proof for Theorem \ref{thm:cospdense}.  We
show that most matrices which are unimodularly equivalent to a matrix
$A\in\MnZ$, are $p$-characterized (and hence $p$-correspondent)
when $p$ is large compared to $n$.
The main tool is the following.

In what follows, $\cont(g)$ denotes the content of a polynomial $g$, that
is, the GCD of the coefficients of $g$.

\begin{lemma}
  \label{lem:content-coefficients}
  Let $A \in \MnZ$ have rank $r$.  Let $\U, \V$ be $\nxn$ matrices whose
  $2n^2$ entries are algebraically independent indeterminates $u_{ij}$
  and $v_{ij}$ respectively.
  Let $g_k$ be the coefficient of $x^{n-k}$ in
  the characteristic polynomial of $B = \U A \V$.
  Then for $k \in [1, r]$, $g_k$ is a polynomial of total degree $2k$
  and $\cont(g_k)$ is $\Delta_k$, the $k$th determinantal
  divisor of $A$.
\end{lemma}
\begin{proof}
  Assume throughout that $k \leq r$. Using the Cauchy-Binet formula,
  \begin{equation}
    \label{eq:fi-cauchy}
    \begin{aligned}
      g_k & = (-1)^k \sum_{\sigma \in \calC_k^n} B \binom{\sigma}{\sigma} =
      (-1)^k \sum_{\sigma,
        \tau, \omega \in \calC_k^n} \U \binom{\sigma}{\tau} A
      \binom{\tau}{\omega} \V \binom{\omega}{\sigma} \\
      & = (-1)^k \sum_{\tau, \omega \in \calC_k^n} A
      \binom{\tau}{\omega} \Upsilon_{\tau,\omega},
      ~~~\mbox{where}~~~
    \Upsilon_{\tau,\omega}=
    \sum_{\sigma\in\calC_k^n}
    \U \binom{\sigma}{\tau}\V \binom{\omega}{\sigma}.
    \end{aligned}
  \end{equation}

  We first show that $\Upsilon_{\tau,\omega}$ has content $1$.  By
  Leibniz's determinant expansion on the minor of $\U$ selected
  by the first $k$ rows, and the columns
  given by the indices in $\tau\in\calC_k^n$,  we have
  \begin{align*}
    \U\binom{(1,2,...,k)}{\tau} & = \sum_{\mu\in S_k} \sgn(\mu) \prod_{1\leq i\leq k} u_{\mu_i,\tau_i}\\
    & = u_{1,\tau_1}u_{2,\tau_2}\cdots
    u_{k,\tau_k} + \sum_{\substack{\mu\in S_k\\\mu\neq
        \textrm{id}}}\prod_{1\leq i\leq k} \sgn(\mu)\, u_{\mu_i,\tau_i},
    \end{align*}
  where $S_k$ is the symmetric group of permutations of $k$ symbols,
  $(\mu_1,\ldots,\mu_k)$ is a permutation of $\{1,\ldots,k\}$ and
  $\mbox{id}=(1,\ldots,k)$ is the identity permutation.
  Similarly,
  \[
    \V\binom{\omega}{(1,\ldots,k)}  =
    v_{\omega_1,1}v_{\omega_2,2}\cdots
    v_{\omega_k,k} + \sum_{\substack{\mu\in S_k\\\mu\neq
        \textrm{id}}}\prod_{1\leq i\leq k} \sgn(\mu)\, v_{\omega_i,\mu_i}.
    \]
    We observe that $\U\binom{(1,2,...,k)}{\tau}$ contains the
    distinguished monomial
    $u_{1,\tau_1}\cdots u_{k,\tau_k}$ which is not found
    in any of the remaining terms of the expansion of
    $\U\binom{(1,\ldots,k}{\tau}$ and hence has coefficient $1$ (since each
    permutation $\mu$ is distinct), and is not found in the expansion
    of $\U\binom{\sigma'}{\tau'}$ for any other
    $\sigma',\tau'\in\calC^n_k$ (since the variables in the term allow
    us to identify the subsets $\sigma'$ and $\tau'$).  Similarly,
    $\V\binom{\omega}{(1,\ldots,k)}$ contains the distinguished monomial
    $v_{\omega_1,1}\cdots v_{\omega_k,k} $ with
    coefficient 1 which is not found in $\V\binom{\omega'}{\sigma'}$
    for any other $\omega',\sigma'\in\calC^n_k$.

    Thus, for every choice of $\tau,\omega$, the polynomial
    $\Upsilon_{\tau,\omega}$ has a monic
    distinguished term
    $u_{1,\tau_1}\cdots u_{k,\tau_k}v_{\omega_1,1}\cdots
    v_{\omega_k,k} $
    not appearing in $\Upsilon_{\tau',\omega'}$ for any other
    $\tau',\omega'\in\calC_k^n$.  Thus $\Upsilon_{\tau,\omega}$ is
    non-zero, has degree $2k$, and has content 1.

    It follows immediately that $g_k$ has degree $2k$ and content
    which is the GCD of all $A \binom{\tau}{\omega}$, which is
    precisely $\Delta_k$.
\end{proof}

A related result is found in \cite[Theorem 1.4]{Gie01}.  A similar
technique is used in \citep[Theorem~2]{Kaltofen:1991}, where a minor
with symbolic entries is explicitly selected and shown to be
lexicographically unique and hence the resulting polynomial,
e.g. $g_k$, is shown to be non-zero.

The following lemma is used to count the number of matrices with a given
property. While this result resembles the well-known Schwartz-Zippel lemma
\citep{Zippel:1979, Schwartz:1980}, similar statements can be traced to earlier
literature, for example in \citep{Kasami:1968}.

\begin{lemma}\label{prop:schwartz-zippel}
  Let $p$ be a prime, $\ell \ge 1$ be an integer, and $g \in \ZZ[x_1, \ldots,
    x_n]$ be a non-zero polynomial of total degree $k$.  Then the number of
  points $\alpha = (\alpha_1, \ldots, \alpha_n) \in [0, \ell p)^n$ for which
    $g(\alpha) \equiv 0 \pmod{p}$ is at most $\ell^n k p^{n - 1}$.
\end{lemma}
\begin{proof}
  As a shorthand, we call $\alpha \in \ZZ^n$ a \emph{$p$-root} if $f(\alpha)
  \equiv 0 \pmod{p}$. For $\ell = 1$ the statement of the lemma becomes exactly
  Corollary~1 of \citep{Schwartz:1980}: the number of $p$-roots in the cube $[0,
  p)^n$ is at most $k p^{n - 1}$.

  Now assume $\ell > 1$. Every $p$-root $b \in [0, \ell p)^n$ can be written
  with component-wise Euclidean division as $(b_1, \ldots, b_n) = (\alpha_1 +
  r_1 p, \ldots, \alpha_n + r_n p) = \alpha + (r_1 p, \ldots, r_n p)$ where $r_i
  \in [0, \ell - 1)$ and $\alpha = (\alpha_1, \ldots, \alpha_n) \in [0,
  p)^n$. Then $\alpha$ must be a $p$-root because $b \equiv \alpha \pmod{p}$.
  Conversely if $\alpha = (\alpha_1, \ldots, \alpha_n) \in [0, p)^n$ is a
  $p$-root, then $(\alpha_1 + r_1 p, \ldots, \alpha_n + r_n p) \in [0, \ell
  p)^n$ is a $p$-root for all the $\ell^n$ possible values of $(r_1, \ldots,
  r_n) \in [0, \ell)^n$. Thus there are at most $\ell^n \cdot k p^{n-1}$
  $p$-roots in the cube $[0, \ell p)^n$.
\end{proof}


\begin{lemma}
  \label{lem:uav-randomized}
  Let $A\in\MnZ$, $\epsilon>0$, $p$ a prime greater than $(n^2+3n)/\epsilon$,
  and $N$ a non-zero integer divisible by $p$. The number of pairs of
  matrices $(U, V)$ with entries from
    $[0, N)$ such that $U$ and $V$ are both non-singular modulo $p$, and that
    $UAV$ is $p$-characterized, and hence $p$-correspondent, is at least $(1 -
    \epsilon) N^{2n^2}$.
\end{lemma}
\begin{proof}
  We show this count by associating each pair of matrices $(U, V)$ with
  a point in
  $[0, N)^{2n^2}$ and then bounding the number of roots of a particular set of
  polynomials
  when evaluated in the cube $[0, N)^{2n^2}$.

  First consider the product
  $\U A \V$ where $\U, \V$ have symbolic independent indeterminates $u_{ij}$
  and $v_{ij}$ for all $i, j \in [1, n]$.
  Let the characteristic polynomial of $\U A \V$ be
  \[
     g = x^n +  g_1 x^{n-1} + \ldots + g_k x^{n-k} + \ldots + g_n
  \]
  Then each
  \[
     \gbar_k = \frac{g_k}{\Delta_k(A)} \in \ZZ[u_{11}, u_{12} ,\ldots, v_{nn}]
  \]
  is a polynomial in the entries of $\U, \V$ with degree $2k$ and content $1$
  by Lemma~\ref{lem:content-coefficients}

  Each pair of matrices $U, V$ in the lemma statement defines
  a point in $[0, N)^{2n^2}$; the entries of $U, V$ define the
  values for the $2n^2$ variables $u_{ij}$ and $v_{ij}$.
  The coefficients of the characteristic polynomial of
  each matrix $U A V$ is obtained by evaluating the polynomials
  $g_k$ at the point in $[0, N)^{2n^2}$ defined by $(U, V)$.
  Then using Lemma~\ref{prop:schwartz-zippel}, we have
  $\gbar_k \equiv 0 \pmod{p}$ in at most
  $(N/p)^{2n^2} \cdot 2k p^{2n^2-1} = N^{2n^2} \cdot 2k / p$ points.

  The determinant of $\U$ (resp. $\V$) is a polynomial of degree $n$ in all of the
  $2n^2$ variables $u_{ij}$ (resp. $v_{ij}$), and hence $\det U \equiv
  0 \pmod{p}$ in at most
  $(N/p)^{2n^2} n p^{2n^2-1} = N^{2n^2} n / p$ points in the cube $[0, N)^{2n^2}$
  by Lemma~\ref{prop:schwartz-zippel}.

  Thus the number of points in $[0, N)^{2n^2}$ for which
  $\det U \equiv 0 \pmod{p}$ or $\det V \equiv 0 \pmod{p}$,
  or that $\gbar_k \equiv 0 \pmod{p}$ for \emph{some}
  $k\in [1,r]$ is at most
  \[
  \frac{2n {N^{2n^2}}}{p} + \sum_{1\leq k\leq r} \frac{2k
    {N^{2n^2}}}{p} = \frac{2n {N^{2n^2}}}{p} +
  \frac{r(r+1) {N^{2n^2}}}{p} \leq \frac{(n^2+3n)}{p}
       {N^{2n^2}} < \epsilon {N^{2n^2}}.
  \]

  If all $\gbar_k \not \equiv 0\pmod{p}$ for $k \in [1, r]$, then
  $\ord_p(\gbar_k) = 0$ and
  $\ord_p(g_k) = \ord_p(\Delta_k)$ for $k \in [1, r]$, so $UAV$ is
  $p$-characterized, and hence $p$-correspondent. The number of
    pairs $(U,V)$ for which this holds is then at least $N^{2n^2} - \epsilon
    N^{2n^2} = (1 - \epsilon) N^{2n^2}$.
\end{proof}

\begin{example}
  Intuitively, Lemma~\ref{lem:uav-randomized} shows
  that most choices of the pairs $(U, V)$ will
  result in $UAV$ being $p$-correspondent.
  Consider the matrix:
  \[
  A =
  \begin{pmatrix}
    -48 & -83 & 91 & -497 \\
    -407 & -666 & 637 & -3948 \\
    83 & 125 & -91 & 728 \\
    -291 & -599 & 903 & -3717
  \end{pmatrix},
  \]
  $A$ is not p-correspondent since its invariant factors are
  $[ 1, 7, 7, 49 ]$
  and its $7$-adic eigenvalues are (using the Sage computer algebra system):
  \[
  \begin{aligned}
    6 \cdot 7 & + 7^2 + O(7^3), \\
    3 \cdot 7 & + 3 \cdot 7^2 + O(7^3), \\
    1 \cdot 7 & + 4 \cdot 7^2 + O(7^3), \\
    2 \cdot 7 & + 3 \cdot 7^2 + O(7^3).
  \end{aligned}
  \]
  Now consider a particular choice of $U,V \in \ZZ^{4\times 4}$:
  \[
  U =
  \begin{pmatrix}
    6 & 1 & 0 & 20 \\
    1 & 1 & 1 & 0 \\
    1 & 1 & 1 & 2 \\
    1 & 3 & 0 & 1
  \end{pmatrix}, \quad
  V =
  \begin{pmatrix}
    1 & 1 & 1 & 17 \\
    0 & 0 & 3 & 2 \\
    0 & 5 & 1 & 3 \\
    1 & 0 & 9 & 56
  \end{pmatrix}.
  \]
  and let
  \[
  \Atil = U A V = \begin{pmatrix}
    -87785 & 89700 & -758134 & -4630434 \\
    -4089 & 2813 & -35060 & -213813 \\
    -12105 & 11261 & -104336 & -636989 \\
    -17618 & 12965 & -151217 & -922413
  \end{pmatrix}.
  \]
  Using Sage we can verify
  that $\det U \nequiv 0 \pmod 7$, $\det V \nequiv 0\pmod 7$,
  that the invariant factors of $\Atil$ are
  $[ 1, 7, 7, 2^{10} \cdot 7^2 \cdot 17 ]$
  and that the  $7$-adic valuations of the eigenvalues of $\Atil$
  are $[0, 1, 1, 2]$.
  As expected from Lemma \ref{lem:uav-randomized}, $\Atil$ is
  $p$-correspondent. \qed
\end{example}


\subsection{Density at large primes}

To establish the density of $p$-correspondent matrices, we consider
the set $\SS_S^m$ of all matrices with a given a Smith form $S$ and
integer entries from $[0, p^m)$, and show that
most matrices in this
set are $p$-characterized.

%

We employ the notion of a Smith normal form over the ring $\ZZ/p^m\ZZ$.
We choose $\{0, \ldots, p^m - 1\}$ for representing the residue
  classes in this ring.  For any non-zero $n\times n$ matrix
$A$ over the principal ideal ring $\ZZ/p^m\ZZ$, there exist two unimodular
matrices $U, V \in \GL_n(\ZZ/p^m\ZZ)$ and a unique matrix $S_p = \diag(p^{e_1},
\ldots, p^{e_r}, 0, \ldots, 0)\in(\ZZ/p^m\ZZ)^\nxn$, for integers $0\leq
e_1\leq\cdots\leq e_r<m$, such that $A = U S_p V$.
A matrix $U\in\GL_n(\ZZ/p^m\ZZ)$ is unimodular
if its inverse is also in $\GL_n(\ZZ/p^m\ZZ)$, or equivalently that its
determinant is non-zero modulo $p$.  We call $S_p$ the Smith form of $A$ over
$\ZZ/p^m\ZZ$.  Its existence and uniqueness follows from
\cite{Kaplansky:1949}\footnote{
Kaplansky uses the term \emph{diagonal reduction}
to denote the Smith form diagonalization, and
\emph{elementary divisor ring}
to denote a ring over which every matrix admits a diagonal
reduction (see \S 2, pp. 465).
In the paragraph following Theorem~12.3
he concludes (on pp. 487) that every
commutative principal ideal ring is an elementary divisor
ring.
In Theorem~9.3 he shows the uniqueness
of the invariant factors (up to associates)
by establishing an equivalent uniqueness result for modules
rather than matrices. See his argument on pp. 478 for the
equivalence of this result between matrices and modules.
Now observe that $\ZZ/p^m\ZZ$ is a
principal ideal ring to get existence and uniqueness of Smith form
 over this ring. Every ideal in this ring is generated
by a power of $p$ and hence the non-zero invariant factors are powers
of $p$.}.
If
$\Ahat\in\ZZ^\nxn$ is such that $\Ahat\equiv A \pmod{p^m}$, and $\Ahat$ has
integer Smith form $\diag(s_1,\ldots,s_{\rhat},0,\ldots,0)\in\ZZ^\nxn$ then
$r \le \rhat$ and
  $e_i = v_p(s_i) $ for $1\leq i\leq r$.

The following lemma relates the construction $UAV$ in
  Lemma~\ref{lem:uav-randomized} to integer matrices with prescribed $p$-adic
  valuations on their invariant factors.

For any integer $a$ and any prime power $p^m$, we use $a \rem p^m$ to denote the
unique, non-negative, integer $r < p^m$ such that $a = q p^m + r$ for some
integer $q$.  We extend the ``$\rem p^m$'' operator to vectors and matrices
using element-wise application. Note that ``$\rem p^m$'' operator is distinct
from the ``$\textrm{mod}\ p^m$'' equivalence relation; for example, $(a + b)
\rem p^m \neq (a \rem p^m) + (b \rem p^m)$ in general.

\begin{lemma}
  \label{lem:uniform-sampling-1}
  Fix an integer $n$, a prime $p$, and integers
  $0 \leq e_1 \leq \cdots \leq e_n < \infty$, and let
  $m > e_1 + \ldots + e_n$. Let $S = \diag(p^{e_1}, \ldots, p^{e_n})$
  and $\SS_S^m \subseteq \ZZ^\nxn$ as in
  Definition~\ref{def:SS_S^m}. Fix any $A \in \SS_S^m$. Let
  $L, R \in \ZZ^\nxn$ be any integer matrices satisfying
  $A = (LSR) \rem p^m$.  Then
  $v_p(\det L)=v_p(\det R)=0$, and hence
  $L,R$ are both invertible modulo $p^m$.
\end{lemma}
\begin{proof}
    If $A = (LSR) \rem p^m$ then there exists an integer
    matrix $Q$ (whose entries are the element-wise quotients
    of the Euclidean division)
    such that $A + p^m Q = LSR$.
    { Taking the determinants of both sides, we have}
    \[
        \det(A + p^m Q) = \det(L) \det(S) \det(R).
    \]
    Both sides are (products of) determinants, and hence polynomials in the matrix
    entries, and projecting modulo $p^m$ we get
    \[
        \det(A) \equiv \det(L) \det(S)\det(R) \pmod{p^m},
    \]
    or equivalently
    \[
    \det(A) + p^mq = \det(L)\det(S)\det(R),
    \]
    for some $q\in\ZZ$.

    Since $A\in\SS_S^m$ we know that $v_p(\det(A))=v_p(\det(S))$, and
   moreover, $0\leq v_p(\det(A))<m$ by the conditions of the lemma.
    { Thus $v_p(\det(A)+p^mq)= 
    v_p(\det(A))<m$, since the valuation, the number of times $p$ divides
    $\det(A)+p^mq$, is unaffected by the second summand.}
    Taking the valuation of both sides, we then have
    \[
        v_p(\det(A)+p^mq) = v_p(\det(A)) =
        v_p(\det(L)) + v_p(\det(S)) + v_p(\det(R)).
    \]
   Since $0\leq v_p(\det A)=v_p(\det S)<m$, it must be the case that $v_p(\det(L))=v_p(\det(R))=0$.
\end{proof}

\begin{lemma}
  \label{lem:uniform-sampling}
  Fix an integer $n$, a prime $p$, and integers $0\leq e_1\leq \cdots\leq e_n$,
  and let $m > e_1 + \cdots + e_n$.  Let $S=\diag(p^{e_1},\ldots,p^{e_n})$ and
  $\SS_S^m\subseteq\ZZ^\nxn$ as in Definition~\ref{def:SS_S^m}.  Fix any $A \in
  \SS_S^m$. Define
  \[
  P_A =
  \left\{
    (L, R) : L, R \text{ have entries from } [0,
    p^m) \text{ and } A = (LSR) \rem p^m \
  \right\}.
  \]
  Then $|P_A| = |\GL_n(\ZZ/p^m\ZZ)^2| / |\SS_S^m|$,
  independent of the choice of $A$.
\end{lemma}
\begin{proof}
    We have chosen $[0, p^m)$ to represent
    $\ZZ/p^m\ZZ$, so any integer matrix from $[0,p^m)^\nxn$ has a unique
    image over $\ZZ/p^m\ZZ$ and vice versa.
    {To keep track of the rings we are working on,
    we use the subscript ${p^m}$ to
    denote matrices over the ring $\ZZ/p^m\ZZ$.}
    We first show that there is a bijection
    between $P_A$ and
    \[
       P'_A = \{ (L_{p^m}, R_{p^m}) \in\GL_n(\ZZ/p^m\ZZ)^2:
         A_{p^m} \equiv L_{p^m} S_{p^m} R_{p^m} \pmod{p^m} \}.
    \]
    If $(L, R) \in P_A$, and its image over $\ZZ/p^m\ZZ$ is
    $(L_{p^m}, R_{p^m})$, then
    $(L_{p^m}, R_{p^m}) \in\GL_n(\ZZ/p^m\ZZ)^2$
    by Lemma~\ref{lem:uniform-sampling-1}.
    {
    Also, $A = (LSR) \rem p^m$ implies that $A + p^m Q = LSR$
    for some integer matrix $Q$
    and so $A_{p^m} \equiv L_{p^m} S_{p^m} R_{p^m} \pmod{p^m}$.
    Thus $(L_{p^m}, R_{p^m}) \in P_A'$.}

    Conversely, let $(L_{p^m}, R_{p^m}) \in P_A'$ and their preimages
    be $L, R \in [0, p^m)^\nxn$. The equivalence
    $A_{p^m} \equiv L_{p^m} S_{p^m} R_{p^m} \pmod{p^m}$ implies
    \[
    A + p^m Q_1 = (L + p^m Q_2) (S + p^m Q_3) (R + p^m Q_4),
    \]
    for some integer matrices $Q_1, Q_2, Q_3, Q_4$. This can be simplified to
    \[
    A + p^m Q_5 = L S R,
    \]
    for some integer matrix $Q_5$. In other words,
    \[
    A = (LSR) \rem p^m,
    \]
    and so $(L, R) \in P_A$. Thus there is a bijection between $P_A$ and $P_A'$.

    We now observe that the multiplicative group $\GL_n(\ZZ/p^m\ZZ)^2$
    acts on $(\ZZ/p^m\ZZ)^\nxn$ via left and right multiplication:
    $(L_{p^m}, R_{p^m}) \in \GL_n(\ZZ/p^m\ZZ)^2$ acts on
    $A_{p^m} \in (\ZZ/p^m\ZZ)^\nxn$ to produce
    $L_{p^m} A_{p^m} R_{p^m} \in (\ZZ/p^m\ZZ)^\nxn$.  Then
    $\orbit(A_{p^m}) = \orbit(S_{p^m})$ under this group action since
    there exists at least one such $L_{p^m}, R_{p^m}$ with
    $L_{p^m} A_{p^m} R_{p^m} \equiv S_{p^m} \pmod{p^m}$.  Furthermore,
    the orbit of $S_{p^m}$ corresponds to $\SS_S^m$: every matrix in
    $\SS_S^m$ has a natural image over $\ZZ/p^m\ZZ$ which can be
    written as $L_{p^m} S_{p^m} R_{p^m} \pmod{p^m}$ for suitable
    choice of $L_{p^m}, R_{p^m} \in \GL_n(\ZZ/p^m\ZZ)$, and conversely
    every matrix $L_{p^m} S_{p^m} R_{p^m} \pmod{p^m}$ corresponds to a
    preimage integer matrix in $\SS_S^m$. Therefore we know
    $|\orbit(S_{p^m})| = |\SS_S^m|$.

   Let $\stab(S_{p^m})$ be the stabilizer of $S_{p^m}$ defined as:
   \[
    \left\{ (L_{p^m}, R_{p^m}) : L_{p^m}, R_{p^m} \in (\ZZ/p^m\ZZ)^\nxn,
     S_{p^m} \equiv L_{p^m} S_{p^m} R_{p^m} \pmod{p^m} \right\},
   \]
   and let $A_{p^m} \equiv U_{p^m} S_{p^m} V_{p^m} \pmod{p^m}$
   be a Smith decomposition of $A_{p^m}$, then every
   pair $(L_{p^m}, R_{p^m}) \in P'_A$ can be mapped to a pair
   $(U_{p^m}^{-1} L_{p^m}, R_{p^m} V_{p^m}^{-1}) \in \stab(S_{p^m})$.
   Similarly, every pair $(L_{p^m}, R_{p^m}) \in \stab(S_{p^m})$
   can be mapped to a pair $(U_{p^m} L_{p^m}, R_{p^m} V_{p^m}) \in P'_A$.
   Thus $|P'_A| = |\stab(S_{p^m})|$.

   By the orbit-stabilizer theorem \citep[Proposition 7.2]{Artin:1991}, we
   have
   \[
   |\orbit(S_{p^m})| \cdot |\stab(S_{p^m})| = |\GL_n(\ZZ/p^m\ZZ)^2|.
   \]
   The lemma statement follows because
   $|\orbit(S_{p^m})| = |\SS_S^m|$, and $|\stab(S_{p^m})| =
   |P'_A| = |P_A|$.
\end{proof}

\begin{lemma}
  \label{prop:val-phi-rem}
  Let $\phi \in \ZZ[x_1, \ldots, x_\ell]$ be a non-zero polynomial and
  $a_1, \ldots, a_\ell$ $\in \ZZ$. Let $p$ be a prime and $m \ge 1$
  be an integer.
  Let $k = v_p(\phi(a_1, \ldots,
  a_\ell))$ and $\kbar = v_p(\phi(a_1 \rem p^m, \ldots, a_\ell \rem
  p^m))$. Then
  \begin{itemize}[topsep=3pt,itemsep=0pt]
    \item[(i)] If $k < m$ then $\kbar = k$.
    \item[(ii)] If $k \ge m$ then $\kbar \ge m$.
    \item[(iii)] If $k = \infty$ then $\kbar \ge m$.
  \end{itemize}
\end{lemma}
\begin{proof}
  Let $\phi(a_1, \ldots, a_\ell) = p^k \alpha$ for some $\alpha \in \ZZ$ and $p
  \nmid \alpha$. For all $i \in [1,
      \ell]$, apply the Euclidean division to $a_i$ and $p^m$ to get $a_i = r_i
    + p^m q_i$ where $p^m \nmid q_i$ and $r_i = a_i \rem p^m$. Then
  \[ \phi(r_1 + p^m q_1 , \ldots, r_\ell + p^m q_\ell) \equiv \phi(r_1, \ldots,
  r_\ell) \pmod{p^m}. \]

  \noindent (i) If $k < m$ then
  \[ \phi(r_1 + p^m q_1, \ldots, r_\ell + p^m
    q_\ell) \equiv  \phi(r_1, \ldots, r_\ell) \equiv p^k \alpha \pmod{p^m} \]
  and $\phi(r_1, \ldots, r_\ell) = p^k \alpha + p^m u$ for some $u \in \ZZ$.
  Now $v_p(p^k \alpha + p^m u) = k$ since $p^m u$ has valuation at least
  $m > k$. So $\kbar = k$.

  \noindent (ii) If $k \ge m$ then $\phi(r_1 + p^m q_1 , \ldots, r_\ell + p^m q_\ell)
  \equiv \phi(r_1, \ldots, r_\ell) \equiv 0 \pmod{p^m} $, and $\phi(r_1, \ldots,
  r_\ell) = p^{m + j} u_1$ for some $u_1 \in \ZZ$, $p \nmid u_1$ and some $j \ge
  0$. Then $\kbar = m + j \ge m$.

  \noindent (iii) If $k = \infty$ then $\phi(r_1 + p^m q_1 , \ldots, r_\ell + p^m q_\ell)
  = 0$, and $\phi(r_1, \ldots, r_\ell) \equiv \phi(r_1 + p^m q_1 , \ldots,
  r_\ell + p^m q_\ell) \equiv 0 \pmod{p^m}$, which is similar to part (ii).
\end{proof}

\begin{lemma}
  \label{prop:val-gcd-phi}
  Let $\phi_1, \ldots, \phi_r \in \ZZ[x_1, \ldots, x_{\ell}]$ be
  polynomials such that
  \[
  v_p\biggl(\gcd \bigl\{ \phi_1(a_1, \ldots, a_\ell), \ldots, \phi_r(a_1, \ldots,
  a_\ell) \bigr\}\biggr) = k < m.
  \] Then
  \begin{align*}
  v_p \biggl(\gcd \bigl\{ \phi_1( & a_1 \rem p^m, \ldots, a_\ell \rem p^m), \\[-10pt]
  & \ldots, \phi_r(a_1 \rem p^m, \ldots, a_\ell \rem p^m ) \bigr\}\biggr) = k.
  \end{align*}
\end{lemma}
\begin{proof}
  There exists an $i \in [1, r]$ such that $v_p(\phi_i(a_1, \ldots,
  a_\ell)) = k$ whereas for all other $j \in [1, r] \setminus \{ i
  \}$, we have $v_p(\phi_j(a_1, \ldots, a_\ell)) \ge k$ (and possibly
  $\infty$). Then, by Lemma~\ref{prop:val-phi-rem}, $v_p(\phi_i(a_1
  \rem p^m, \ldots, a_\ell \rem p^m)) = k$ while for all $j$,
  $v_p(\phi_j(a_1 \rem p^m, \ldots, a_\ell \rem p^m))$ is either $k$
  or $m$ or higher than $m$ (but not lower than $k$). Thus the
  valuation of the desired gcd is also $k$.
\end{proof}

We now show that if $A$ is non-singular, then the powers of $p$ in the
Smith form of $A$ and $A \rem p^m$ coincide when $m > v_p(\det A)$.

\begin{lemma}\label{lem:snf-rem-large}
  Let $A \in \ZZ^{\nxn}$ be a non-singular matrix, $m > v_p(\det A)$
  and $\Abar=A\rem p^m$.  Suppose the invariant factors of $A$ and
  $\Abar$ are $s_1,\ldots,s_n$ and $\sbar_1,\ldots,\sbar_n$, respectively.
  Then
  $v_p(s_i)=v_p(\sbar_i)$ for $1\leq i\leq n$.
  \end{lemma}
\begin{proof}
  Let $\Delta_i$ and $\Deltabar_i$
  be the $i$th determinantal divisors of $A$ and $\Abar$ respectively.
  We show equivalently that $v_p(\Delta_i)=v_p(\Deltabar_i)$ for
  $1\leq i\leq n$.  Each $\Delta_i$ (resp. $\Deltabar_i$) is the gcd
  of all $i \times i$ minors of $A$ (resp. $\Abar$), where each such
  minor is a polynomial in the $n^2$ entries of $A$ (resp. $\Abar$).
  Then by Lemma~\ref{prop:val-gcd-phi} we have $v_p(\Delta_i) =
  v_p(\Deltabar_i)$ for all $i \in [1, n]$.
\end{proof}

\begin{lemma}
  \label{lem:val-fi}
  Let $A \in \ZZ^{n \times n}$, $\det A \neq 0$ and $m > v_p(\det A)$.  Let
  $f_i^{M}$ denote the $x^{n-i}$ coefficient of the characteristic polynomial of
  a matrix $M$. For all $i \in [1, n]$, if $v_p(f_i^{A}) = k < m$ then
  $v_p(f_i^{A \rem p^m}) = k$.
\end{lemma}
\begin{proof}
  Each $f_i^{A}$ is the sum of all $i \times i$ \emph{principal} minors of $A$,
  which is a polynomial in the entries of $A$.  The claim then follows
  by Lemma~\ref{prop:val-phi-rem}.
\end{proof}

We now apply the above lemmas to get the following.

\begin{lemma}\label{lem:remp-is-p-char}
  Let $A$ be a $p$-characterized non-singular matrix and let $m > v_p(\det
  A)$. Then $\Abar = A \rem p^m$ is also $p$-characterized.
\end{lemma}
\begin{proof}
  Let $\Delta_i$ and $\Deltabar_i$ be the $i$th determinantal divisors
  of $A$ and $\Abar$ respectively, for $1\leq i\leq n$.  If $A$ is a
  $p$-characterized, then $v_p(f_i^{A}) = v_p(\Delta_i)$ for each $i
  \in [1, n]$. By Lemma~\ref{lem:snf-rem-large} and
  Lemma~\ref{lem:val-fi}, we have $v_p(\Deltabar_i) = v_p(\Delta_i)$
  and $v_p(f_i^{\Abar}) = v_p(f_i^{A})$. So $\Abar$ is
  $p$-characterized.
\end{proof}


\begin{example}
  For a prime $p$ consider the matrix $A$ with its Smith form decomposition:
  \[
  A =
  \begin{bmatrix}
    p^3 + 1 & p \\ 2 p^4 & p^2
  \end{bmatrix}
  =
  \begin{bmatrix}
    1 & 0 \\ -p^4 & 1
  \end{bmatrix}
  \begin{bmatrix}
    1 & \\ & - p^2 + p^5
  \end{bmatrix}
  \begin{bmatrix}
    1 & p \\ -p^2 & -1 - p^3
  \end{bmatrix}
  \]
  whose characteristic polynomial is
  \[ f = x^2 - (1 + p^2 + p^3) x + p^2 - p^5. \]
  Observe that $A$ is
  $p$-characterized.
  Now let $m = 3$ and consider $A \rem p^m$ and its Smith form decomposition:
  \[
  A \rem p^3 =
  \begin{bmatrix}
    1 & p \\ 0 & p^2
  \end{bmatrix}
  =
  \begin{bmatrix}
    1 & 0 \\ 0 & 1
  \end{bmatrix}
  \begin{bmatrix}
    1 & \\ & p^2
  \end{bmatrix}
  \begin{bmatrix}
    1 & -p \\ 0 & 1
  \end{bmatrix},
  \]
  which has the characteristic polynomial \[ x^2 - (1 + p^2) x + p^2. \] Thus $A
  \rem p^3$ is $p$-characterized as well. \qed
\end{example}

The following bound is a relatively well-known fact, but we prove
  it for completeness.  Here $\M_n(\ZZ/p^m\ZZ)$ is the ring of
  $n\times n$ matrices over $\ZZ/p^m\ZZ$.

\begin{lemma}\label{lem:density-glnpm}
  $|\M_n(\ZZ/p^m\ZZ)| / |\GL_n(\ZZ/p^m\ZZ)| < 4$.
\end{lemma}
\begin{proof}
  Any matrix $A \in \M_n(\ZZ/p^m\ZZ)$ can be written as $A = A_0 + p A_1 +
  \ldots + p^{m-1} A_{m-1}$ with $A_i$'s having entries from $[0, p)$. Then $A
  \in \GL_n(\ZZ/p^m\ZZ)$ if and only if $A_0 \in \GL_n(\ZZ/p\ZZ)$. There are
  $(p^{n^2})^{m-1}$ ways to construct the components $A_1, \ldots, A_{m-1}$ for
  each given $A_0 \in \GL_n(\ZZ/p\ZZ)$. So $|\GL_n(\ZZ/p^m\ZZ)| = p^{(m-1)n^2}
  \cdot |\GL_n(\ZZ/p\ZZ)|$.  Next, recall the well-known density bound for
  non-singular matrices over finite fields:
  \[
  \frac{ |\GL_n(\ZZ/p\ZZ)| }{ p^{n^2} } = \left(1 - \frac{1}{p} \right) \left(1
    - \frac{1}{p^2} \right) \cdots \left(1 - \frac{1}{p^n} \right) > 1/4.
  \]
  Thus
  \[
    \cfrac{|\M_n(\ZZ/p^m\ZZ)|}{|\GL_n(\ZZ/p^m\ZZ)|} = \cfrac{p^{mn^2}}{
      p^{(m-1)n^2} |\GL_n(\ZZ/p\ZZ)|} = \cfrac{p^{n^2}}{ |\GL_n(\ZZ/p\ZZ)| } < 4.
  \]
\end{proof}

We can now establish our main density result.

\begin{theorem}
  \label{thm:density-class-s-m}
  Let $n$ be a positive integer, $\epsilon > 0$, and $p$ any prime greater than
  $16 (n^2+3n)/ \epsilon$. Fix a set of integers $0 \le e_1 \le e_2 \le
  \cdots \le e_n < \infty$ and let $m \ge e_1 + \ldots + e_n + 1$ and $S =
  \diag(p^{e_1}, \ldots, p^{e_n}) \in \ZZ^{\nxn}$.  Then the number of matrices
  in $\SS_S^m$ which are $p$-characterized and hence
  $p$-correspondent is at least $(1 - \epsilon) \cdot
  |\SS_S^m|$.
\end{theorem}
\begin{proof}
  Let
    \[
    P = \{(L, R) : L, R \in [0, p^m)^{\nxn} \}.
    \]
    For any $A \in \SS_S^m$, let $P_A \subseteq P$ be as in
    Lemma~\ref{lem:uniform-sampling}:
    \[
    P_A = \{ (L, R) : L, R \text{ have entries from } [0, p^m) \text{ and } A
    = (LSR) \rem p^m \ \}.
    \]
    If at least one pair $(L, R) \in P_A$ is such that $LSR$ is
    $p$-characterized, then $A$ is $p$-characterized by
    Lemma~\ref{lem:remp-is-p-char} (recall $A = (LSR) \rem p^m$ and $m \ge e_1 +
    \ldots + e_n + 1$ implies $m > v_p(\det A)$).  On the other hand, if every
    pair $(L,R) \in P_A$ is such that $LSR$ is not $p$-characterized then $A$
    can be either $p$-characterized or not (because the converse of
    Lemma~\ref{lem:remp-is-p-char} is not necessarily true; some non
    $p$-characterized matrices can become $p$-characterized after applying $\rem
    p^m$). To derive an upper bound on the number of non $p$-characterized
    matrices in $\SS_S^m$, we allow the worst outcome: $A = (LSR) \rem p^m$ is not
    $p$-characterized when $LSR$ is not $p$-characterized for all pairs $(L, R)
    \in P_A$.

    The number of sets, $P_A$, having every pair $(L, R)$ with a non
    $p$-characterized product $LSR$, can be obtained as the ratio between the
    total number of pairs giving non $p$-characterized products (which is at
    most $(\epsilon / 16) |P|$ by Lemma~\ref{lem:uav-randomized}) divided by the
    size of each $P_A$ (which is $|\GL_n(\ZZ/p^m\ZZ)^2|/|\SS_S^m|$ by
    Lemma~\ref{lem:uniform-sampling}). So the maximum number of matrices in
    $\SS_S^m$ which are not $p$-characterized is
    \begin{align*}
      \frac{(\epsilon/16) |P|}{|P_A|} &
      =
      \frac{(\epsilon/16) |\M_n(\ZZ/p^m\ZZ)|^2}{|\GL_n(\ZZ/p^m\ZZ)^2|/|\SS_S^m|}
      <  \epsilon |\SS_S^m|,
    \end{align*}
    where the inequality follows using Lemma~\ref{lem:density-glnpm}.

    Hence there are at least $(1 - \epsilon)|\SS_S^m|$ matrices in $\SS_S^m$
    which are $p$-characterized, and each one of those matrices is also
    $p$-correspondent by Theorem~\ref{thm:pdet->pcosp}.
\end{proof}



\subsection{Density at small primes}

The density estimate of Theorem~\ref{thm:density-class-s-m} is limited to large
primes. Hereby we report on experiments with small primes. For a given size $n$
and a prime power $p^m$, we enumerate the set of all $n \times n$
matrices with entries from $[0, p^m)$
and $v_p(\text{determinant}) < m$.
We then count the fraction of matrices which are $p$-correspondent.

\begin{table}[ht]
  \caption{Density (in percentage) of $p$-characterized and $p$-correspondent
  matrices among the set of $n \times n$ non-singular matrices with entries
  from $[0, p^m)$ and $v_p(\det A) < m$.
  See the text for an explanation of the last column.}
  \label{tab:density}
\begin{center}
\begin{tabular}{|c|c|c|c|c|c|}
  \hline
  $p$ & $m$ & $n$ & $p$-characterized & $p$-correspondent & min $p$-char. \\
  \hline
  & 1 & 2 & 56.25 & 81.25 & 33.33 \\
  & 2 & 2 & 53.52 & 80.08 & 33.33 \\
  & 3 & 2 & 53.34 & 80.00 & 33.33 \\
  2 & 4 & 2 & 53.33 & 80.00 & 33.33 \\
  & 1 & 3 & 29.10 & 71.29 & 18.75 \\
  & 2 & 3 & 26.51 & 70.14 & 16.67 \\
  & 1 & 4 & 15.61 & 66.67 & 6.667 \\
  \hline
  3 & 1 & 2 & 67.90 & 90.12 & 62.50 \\
  & 2 & 2 & 67.50 & 90.00 & 50.00 \\
  & 3 & 2 & 67.50 & 90.00 & 50.00 \\
  & 1 & 3 & 45.58 & 86.73 & 42.77 \\
  \hline
  5 & 1 & 2 & 80.16 & 96.16 & 79.17 \\
  & 2 &  2 & 80.13 & 96.15 & 78.96 \\
  \hline
  7 & 1 & 2 & 85.76 & 98.00 & 85.42 \\
  \hline
\end{tabular}
\end{center}
\end{table}

Table~\ref{tab:density} shows the density of $p$-characterized and
$p$-correspondent non-singular matrices for small values of $p, m, n$. The
fourth and fifth columns report the fraction (in percentage) of
$p$-characterized and $p$-correspondent matrices among all $n \times n$
non-singular matrices with entries $[0, p^m)$ and who determinant has
$p$-adic valuation smaller than $m$.
Recall from Example~\ref{ex:converse} that
matrices can be $p$-correspondent but not necessarily
$p$-characterized, thus
the reported $p$-characterized density is
lower than $p$-correspondent density.
 
The sixth column in the table reports the minimum percentage of
$p$-characterized matrices among all Smith forms. Given $p^m$ and $n$, we
consider the set of all $n \times n$ matrices with entries $[0, p^m)$
and $v_p(\text{determinant}) < m$. We partition these matrices
by their Smith forms localized at $p$, where we only care about the powers of
$p$ in the invariant factors and treat the other prime powers as units. For
example, when $p^m = 2^2$ and $n = 2$, we get the following (non-singular) Smith
forms localized at $2$:
\[
    \begin{bmatrix} 1 & \\ & 1 \end{bmatrix},
    \begin{bmatrix} 1 & \\ & 2 \end{bmatrix}.
\]
We then count the fraction of $p$-characterized matrices in each partition and
report the minimum percentage among all partitions.

Finally, the table shows that the density drops as $n$ increases and
as $p$ decreases, which is consistent with the proofs for large
primes. An open and interesting question is to prove similar
density estimates for small primes,
i.e. when $p$ is small compared to $n$.

\section*{Acknowledgements}
The authors are supported by the
Natural Sciences and Engineering Research
Council of Canada.
Computations were carried out using the Sage computer
algebra system \citep{sage}.
We thank Dino J. Lorenzini and the anonymous referee for
providing helpful comments on earlier versions of this paper.

\bibliographystyle{plainnat}
\bibliography{ref}

\end{document}